\documentclass{amsart}

\usepackage{amsmath,amsthm,amssymb, epsf, graphics,  color, epsfig}
\usepackage{psfrag}
\usepackage{hyperref}
\numberwithin{equation}{section}
\newtheorem {theorem}{Theorem}
\newtheorem  {lemma}{Lemma}[section]
\newtheorem {proposition}[lemma]{Proposition}

\newtheorem{definition}[lemma]{Definition}
\newtheorem {cor}[lemma]{Corollary}

\newtheorem{result}{Fact}[section]
\newcommand{\Sl}{\ensuremath{\mathrm{SL}(2,\mathbb R)}}
\newcommand{\Slz}{\ensuremath{\mathrm{SL}(2,\mathbb Z)}}

\newcommand{\gsl}{\Gamma\backslash\Sl}
\newcommand{\tf}{\tilde{f}}
\newcommand{\gh}{\Gamma\backslash\mathbb H}

\numberwithin{equation}{section}
\newcommand{\sgn}{\textrm{sgn}}

 \newcommand{\Z}{\mathbb{Z}}

\newcommand{\re}{\textrm{Re}}
\newcommand{\im}{\textrm{Im}}
\newcommand{\one}{\frac{1}{2}}

\begin{document}
\title{Rapid computation of $L$-functions for modular forms}
\author{Pankaj Vishe}
\begin{abstract} Let $f$ be a fixed (holomorphic or Maass) modular cusp form, with $L$-function
$L(f,s)$. We describe an algorithm that computes the value
$L(f,1/2+ iT)$ to any specified precision in time $O( 1+|T|^{7/8})$. \end{abstract}
\maketitle
\setcounter{tocdepth}{1}
\tableofcontents
\section{Introduction}

In this paper, we consider the problem of finding a fast algorithm to compute the $L$-function of a (holomorphic or Maass)  cusp form at a point on the critical line. We also discuss the problem of computing the large Fourier coefficients of $f$ at a cusp. We use very similar algorithms to solve both problems.

Let $\Gamma$ be a lattice in $\Sl$. Let $f$ be a cusp form on $\gh$. Specifying $\Gamma$ and $f$ requires (a priori) an infinite amount of data. Therefore
we discuss  what we mean {\em computationally} when we say `given $f$ and a cofinite volume subgroup $\Gamma$ of $\Sl$' in \ref{input}.

The main results can be summarized into the following two theorems:

\begin{theorem}
\label{main theorem 2}
Let $f$ be a holomorphic or Maass cusp form for a congruence subgroup of $\mathrm{SL}(2,\Z)$,  and let $\gamma, \epsilon$ be given fixed positive reals. Then for any $T>0$, we can compute $L(f,1/2+ iT)$  up to an error of $O(T^{-\gamma})$ in time $O(1+T^{\frac{7}{8}+\epsilon})$ using numbers of $O(\log T)$ bits.  The constants involved in $O$ are polynomials in $\frac{1+\gamma}{\epsilon}$ and are independent of $T$.
\end{theorem}

\begin{theorem}
\label{main theorem 1}
Given a lattice $\Gamma \leqslant \Sl$ containing $\left(\begin {array}{cc} 1&1\\0&1\\ \end{array} \right)$, a (holomorphic or Maass) cusp form $f$ on $\Gamma \backslash \mathbb H$,   positive reals $\gamma, \epsilon$, and a positive integer $T$,
  the $T^{th}$ Fourier coefficient of $f$ at any cusp can be computed up to an error of $O(T^{-\gamma}) $ in time $O( T^{7/8+\epsilon}) $ using numbers of $O(\log T)$ bits. The constants involved in $O$ are polynomials in $\frac{1+\gamma}{\epsilon}$ and are independent of $T$.
\end {theorem}

Our algorithm in theorem \ref{main theorem 2} is considerably faster than the current approximate functional equation based algorithms which need the time complexity $O(T^{1+o(1)})$. Our method uses a form of ``geometric approximate functional equation'' (see \eqref{geomfehol} and \eqref{geomfemaass}). It can also be used to compute derivatives of $L(f,1/2+iT)$ because similar ``geometric approximate functional equations'' can be obtained in these cases.

\subsection{Historical background and applications}

The problem of ``computing'' values of the zeta function effectively goes as far back as Riemann. Riemann used the Riemann Siegel formula to compute values of the zeta function and verify the Riemann hypothesis for first few zeroes. The Riemann Siegel formula writes $\zeta(1/2+iT)$ as a main sum of length $O(T^{1/2})$ plus a small easily ``computable'' error. The first practical improvement for this method was given by Odlyzko and Sch\"{o}nhage in \cite{OS}. This algorithm can be used to evaluate $O(T^{1/2})$ values of $\zeta(1/2+it) $ for $t\in [T,T+T^{1/2}]$, to an error of $O(T^{-\gamma})$ each , using $O(T^{1/2+o(1)})$ arithmetic operations on numbers of $O(\log T)$ bits and using $O(T^{1/2+o(1)})$ storage. This method is extremely useful for computing large values of $\zeta(1/2+iT)$ but it does not improve the time required to evaluate a single value of $\zeta$ function. 

Sch\"{o}nhage gave an algorithm to compute $\zeta(1/2+iT)$ in $O(T^{3/8+o(1)})$ in \cite{schon}. Heath-Brown improved the exponent to $1/3$. Hiary gave an explicit algorithm for the exponent $1/3$ using rapid computation of the truncated theta sums in \cite{hiary}. The exponent was later improved by Hiary himself to $4/13$ in \cite{hiary1}. The main idea behind these algorithms is to start with the main sum in the Riemann Siegel formula, cut the sum into ``large'' number of subsums having ``large enough'' lengths and try to compute these subsums rapidly. Fast numerical evaluation of $\zeta(1/2 + it)$ was also considered by Turing \cite{turing}, Berry and Keating \cite{berry}, Rubinstein \cite{rubinstein} and Arias De Reyna \cite{arias} \textit{et al.}

In the case of higher rank $L$-functions, the analogue of the Riemann Siegel formula is given by the approximate functional equation. A detailed description of ``the approximate functional equation'' based algorithms is given by Rubinstein in \cite{rubinstein}. Booker gave an Odlyzko and Sch\"{o}nhage type algorithm to compute multiple valuations of $L(1/2+iT)$ in \cite{booker1}. However, algorithms based on the approximate functional equation remain the fastest way to
compute a {\em single} value of $L(1/2+iT)$, when $T$ is large. In the case of $L$-function associated to a modular (holomorphic or Maass) form, these algorithms have $O(T^{1+o(1)})$ time complexity.

Our algorithm in theorem \ref{main theorem 2} is the first known improvement of the approximate functional equation based algorithms in the $\mathrm{GL}_2$ setting. 

Computing values of $L$-functions on the critical line has various applications in number theory. It can be used to verify the \textit{Generalized Riemann Hypothesis} numerically. It has also been used to connect the distribution of values of $L$-functions on the critical line to the distributions of eigenvalues of unitary random matrices via the recent random matrix theory conjectures. 

The problem of computing $L$-functions is closely related to the problem of finding subconvexity bounds for the $L$-functions. More generally, improving the ``square root of analytic conductor bounds'' coming from the approximate functional equation is of great interest to analytic number theorists. 

Our algorithm starts with writing $L(f,1/2+iT)$ as an integral over the cycle $\{(-1+i/T)t\}$. The equidistribution of this cycle in $\Slz\backslash\mathbb{H}$ is closely related with the moment $\int_{-T}^T|\zeta(1/2+it)|^4dt$ (see \cite{sarnak}) and with subconvexity bounds for the $\mathrm{GL}_2$ $L$-functions.  The results in this paper are closely related to the work of Venkatesh \cite{venkatesh}, where he writes special values of $L$-functions as period integrals over certain cycles. He then uses the equidistribution properties of these cycles to get subconvexity bounds for the $L$-functions. We however use the slow divergence of the cycle $\{(-1+i/T)t\}$ to get a fast way of numerically computing $L(f,1/2+iT)$, up to any given precision.  

In the present paper, we have only considered the $\mathrm{GL}_2$ case. It will be of great interest to generalize the method in this paper to higher rank $L$-functions . A similar (to \ref{geomfemaass}) form of geometric approximate functional equation has been used by Sarnak in \cite[section 3]{sarnak}, in the number field setting. Our method thus could also generalize for $L$-functions corresponding to the Maass forms for $\mathrm{SL(2,}\mathfrak o)$, where $\mathfrak o$ is the ring of integers for a number field $K$. A more interesting problem will be to generalize our technique in the $\mathrm{GL}_n$ setting, where the subconvexity bounds are also not known.

\subsection{Outline of the proof}
\label{outline}
We use a geometric method:
a suitable integral representation of the $L$-function reduces the problem to computing an integral of $f$ (or rather, $\tilde{f}$, the lift of $f$ to $\gsl$ defined by \ref{tf} ) over an approximate horocycle of length $T$.

We then  cut the integral over the approximate horocycle into a sum of integrals over smaller segments of equal length $M$. We sort the starting points of the segments into sets $S_{1},S_{2},...$ such that the points in each set are very close to each other. The key point is that for each $i$, we can compute {\em all the integrals over all segments with starting points in $S_i$} ``in parallel.''

To accomplish this, we use the critical fact that if two points $x_1, x_2$ in $\gsl$ are sufficiently close to one other, then the images $x_1(t), x_2(t)$ of these points
under the time $t$ horocycle flow stay very close to each other ``for a long time.'' The approximate horocycles also have a similar property. We quantify these in lemmas \ref{epsilon'} and \ref{epsilon''} later.
This is a very special property of the horocycle flow. The geodesic flow, for example, does not satisfy such a property.

We return to the description of how to compute the integrals over various segments ``in parallel.''
Consider $S_1$, for instance.  Let $I$ be a fixed segment with starting point in $S_1$ and let $J$ be any other segment with starting point in $S_1$. Let $F(J)$ denote the value of the integral over the segment $J$. We use lemmas \ref{epsilon'}, \ref{epsilon''} to compute $\tf$ at each point on $J$ using a power series expansion around a corresponding point on $I$. In the case of $L$-value computations, this converts $F(J)$ into sum:
\begin{equation}
F(J)=\sum_{i,j} a_{ij}(J)\int_I t^i f_{ij,I}(t)dt
\label{sumFF}
\end{equation}
Here, $f_{ij,I}(t) $ are smooth functions independent of $J$ and $a_{ij}$ are precomputable constants depending on $J$ and the sum over $i$ and $j$ is over a set of size $O(1)$. We can compute each integral on right hand side of \eqref{sumFF} using $O(M^{1+o(1)})$ operations. Therefore, by precomputing certain constants, we are able to
compute all the $J$-integrals in at most a further $O(|S_1|)$ time.

 This grouping leads to a speed-up by a factor that is roughly the size of the groups $|S_i|$. In practice, we cannot make these groups very large, because the amount of time for which the points $x_1(t)$ and $x_2(t)$ stay ``very close'' to each other depends on how close the points $x_1$ and $x_2$ were. Hence, if we try to make the groups ``too large", the ``admissible'' length of the segments decreases, resulting in an increase in the number of segments. This increases the running time.

\subsection{Model of computation}
In practice, specifying a real number completely requires an infinite amount of data. Hence to keep the algorithms simpler and clearer, we will use the real number (infinite precision) model of computation that uses real numbers with error free arithmetic having cost as unit cost per operation. An operation here means addition, subtraction, division, multiplication, evaluation of logarithm (of a complex number $z$ such that $|\arg (z)|<\pi$) and exponential of a complex number.

Our algorithm will work if we work with numbers specified by $O(\log T)$ bits. This will at most add a power of $\log T$ in the time complexity of the algorithm. We refer the readers to \cite[Chapter 7]{vishe} for a brief discussion of the floating point error analysis for this algorithm if we use numbers specified by $O(\log T) $. We refer the readers to \cite[Chapter 8]{traub} and \cite{traub1} for more details about the real number model of computation.

\subsection{Outline of the paper}

A brief account of the notations used in this paper is given in section \ref{section 2}. The algorithm uses a type of ``geometric approximate functional equation''. It is discussed in detail in section \ref{geometric}. The lemma \ref{final bound} is the main tool used in deriving the ``geometric approximate functional equation'' for the Maass form case. We will give a proof of it in appendix \ref{Bessel}. We will give the proof of theorems \ref{main theorem 1} and \ref{main theorem 2} in sections \ref{section3} and \ref{section4} respectively. We will discuss the proof of the ``special property'' of the horocycle (quantified in lemmas \ref{epsilon'}, \ref{epsilon''}) and using it to simplify the integral (quantified in lemmas \ref{lemma 1}, \ref{power series simplification}) in section \ref{epsilonpower} in detail. In section \ref{numerical integration}, we will give a simple algorithm for numerical integration of ``well behaved'' real analytic functions.

\subsection{Acknowledgements}
I would like to thank Akshay Venkatesh for suggesting me this problem, as well as for all the help and support he gave me. I learnt a great deal about this subject from him. I would also like to thank D. Bump, K. Soundararajan, M. Rubinstein, P. Michel, A. Booker and A. Str\"{o}mbergsson for the encouragement as well as for helpful discussions.  I like to thank F. Brumley, G. Hiary and H. Then for providing helpful references. I also like to thank M. Raum and A. Parthasarathy for helpful editorial comments.

Majority of the work in the paper was carried out during my PhD at the Courant Institute and Stanford University, and also during my stay EPFL and MPIM, Bonn. I am very thankful to these institutions for allowing me to visit and carry out my research.
\section{Preliminaries}
\label{section 2}
\subsection{General notation} Throughout, let $\Gamma$ be a lattice in $\Slz$, unless specified otherwise. In the proof theorem \ref{main theorem 1}, $\Gamma$ will denote a lattice in $\Sl$ containing $\left(\begin {array}{cc} 1&1\\0&1\\ \end{array} \right)$.

Throughout, let $T$ be a positive real. In section \ref{section3}, however $T$ will be a positive integer.
We will denote the set of non negative integers by $\mathbb Z_+ $ and the set of non-negative real numbers by $\mathbb R_+$.

We will use the symbol $\ll$ as is standard in analytic number theory: namely, $A\ll B$ means that there exists a positive constant $c$ such that $A\leq cB $. These constants will always be independent of the choice of $T$.

We will use the following special matrices in $\Sl$ throughout the paper:
\begin{align}
n(t)=\left(\begin {array}{cc} 1&t\\0&1\\ \end{array} \right), a(y)=\left(\begin {array}{cc} e^{y/2}&0\\0&e^{-y/2}\\ \end{array} \right),K(\theta)=\left(\begin {array}{cc} \cos \theta&\sin \theta\\-\sin \theta&\cos \theta\\ \end{array} \right).
\end{align}

$e(x)$ will be used to denote $\exp(2\pi i x) $.

Let $f$ be a cusp (Maass or holomorphic) form of weight $k$ on $\gh $. The weight corresponding to a Maass form will be 0. Let us  define a lift $\tf$ of $f$ to $\gsl$ by $\tf: \gsl\rightarrow \mathbb C$ such that \begin{equation}\label{tf}\tf\left(\left(\begin {array}{cc}a&b\\c&d \end{array} \right)\right)=(ci+d)^{-k}f\left(\frac{ai+b}{ci+d}\right). \end{equation}

Let $x$ be any element of $\Sl$. We will frequently abuse the notation to treat $x$ as an element of $\gsl$. \textit{i.e.} we will often denote the coset $\Gamma x$ simply by $x$.

\subsection{Real analytic functions on $\gsl$}
\label{real analytic notation}
Let $x$ be an element of $\Sl$ and let $g$ be a function on $\gsl$, \textit{a priori} $g(x)$ does not make sense but throughout we abuse the notation to define  $$g(x)=g(\Gamma x). $$ \textit{i.e.}  $g(x)$ simply denotes the value of $g$ at the coset corresponding to $x$.

Let $\phi$ be the Iwasawa decomposition given by $$\phi:(t,y,\theta)\in\mathbb R\times \mathbb R \times \mathbb R\rightarrow n(t)a(y)K(\theta).$$ Recall that $\phi$ restricted to the set $\mathbb R\times \mathbb R \times (-\pi,\pi]$ gives a bijection with $\Sl$.
\begin{definition}
\label{U definition}
Given $\epsilon>0$, let $\mathfrak U_{\epsilon}=(-\epsilon,\epsilon)\times (-\epsilon,\epsilon)\times (-\epsilon,\epsilon)$ and $U_\epsilon=\phi(\mathfrak U_\epsilon)\subset \Sl$.
\end{definition}

Let us define the following notion of ``derivatives'' for smooth functions on $\gsl$:

\begin{definition}
\label{liederivative}
Let $g$ be a function on $\Sl$ and $x$ any point in $\Sl$. We define (wherever R.H.S. makes sense)
\begin{align*}
\frac{\partial}{\partial x_1} g(x)&=\frac{\partial}{\partial t}|_{t=0}g(xn(t));\\
\frac{\partial}{\partial x_2} g(x)&=\frac{\partial}{\partial t}\mid_{t=0}g(xa(t));\\
\frac{\partial}{\partial x_3} g(x)&=\frac{\partial}{\partial t}\mid_{t=0}g(xK(t)).
\end{align*}
\end{definition}
Sometimes, we will also use $\partial_i $ to denote $\frac{\partial}{\partial x_i} $.

Given $\beta=(\beta_1,\beta_2,\beta_3)$, let us define $\partial^\beta g(x)$ by \begin{equation}
\partial^\beta g(x)=\frac{\partial^{\beta_1}}{\partial x_1^{\beta_1}}\frac{\partial^{\beta_2}}{\partial x_2^{\beta_2}}\frac{\partial^{\beta_3}}{\partial x_3^{\beta_3}}g(x).
\label{partialbeta}
\end{equation}

For $\beta$ as above we will define $$\beta!=\beta_1!\beta_2!\beta_3!$$ and $$|\beta|=|\beta_1|+|\beta_2|+|\beta_3|. $$
We now define the notion of real analyticity as follows:

A function $g$ on $\gsl$ will be called real analytic, if given any point $x$ in $\gsl$, there exists a positive real number $r_x$ such that $g$ has a power series expansion given by

\begin{equation}
g(xn(t)a(y)K(\theta))=\sum_{\beta=(\beta_1,\beta_2,\beta_3)\in\mathbb Z_+^3}\frac{\partial^\beta g(x)}{\beta!}t^{\beta_1}y^{\beta_2}\theta^{\beta_3}
\label{power}
\end{equation}
for every $(t,y,\theta)\in \mathfrak U_{r_x} $.

Let us use the following notation for the power series expansion.
\begin{definition}
\label{power series}
Let $y,x\in \Sl$ and $t,y,\theta$ be such that $y=xn(t)a(y)K(\theta) $ and $(\beta_1,\beta_2,\beta_3)=\beta\in \mathbb Z_+^3$ define $$(y-x)^\beta =t^{\beta_1}y^{\beta_2}\theta^{\beta_3} .$$
\end{definition}

Hence we can rewrite the Equation \eqref{power} as $$g(y)=\sum_{\beta=(\beta_1,\beta_2,\beta_3),\beta\in \mathbb Z_+^3} \frac{\partial^\beta g(x)}{\beta!}(y-x)^\beta.$$

Throughout, we will assume that for a cusp form $f$, all the derivatives of the lift $\tf$ (of $f$) are bounded uniformly on $\gsl$ by 1. In general it can be proved that given a cusp form $f$, there exists $R$ such that $||\partial^\beta \tf||_\infty \ll R^{|\beta|}$, see \cite[section 8.2]{vishe}. The case when $R>1$ can be dealt with analogously. The assumption that all derivatives are bounded by 1, allows the proofs to be marginally simpler.

\subsection{Input for the algorithm}
\label{input}
As mentioned before, to specify $\Gamma$ and $f$ completely, \textit{a priori} we need an infinite amount of data. Throughout the paper we will assume that $\Gamma$ is given in terms of a finite set of generators. Similarly we assume that modular(holomorphic or Maass) cusp form $f$ on $\gh$ is given by first $O(\log T)$ of its fourier coefficients.

\subsection{Computing values of $\tf$}
In practice, values of $f$ and its derivatives can be computed very rapidly. See \cite{booker} for a method for effectively computing $f$.  It can be easily shown that given any $x$ and any fixed $\gamma$, one can compute $\partial^\beta\tf(x)$ up to the error $O(T^{-\gamma})$ in  $O(T^{o(1)})$ time. Here the constant involved in $O$ is a polynomial in $|\beta|$ and $\gamma$. In this algorithm we only compute values of $\partial^\beta \tf(x)$ for $|\beta|\ll 1$. Allowing $O(T^{o(1)})$ time for each valuation of $\tf$ does not change the time complexity of the algorithm. See \cite[chapter 7]{vishe} for explicit details about it. Hence for simplicity we will assume that each value of $f$ (or $\tf$) or any of it's derivative can be computed exactly in time $O(1)$.

\section{A `geometric' approximate functional equation}
\label{geometric}
 Most existing algorithms to compute a general $L$-function start with an approximate functional equation. This is a generalization of the Riemann- Siegel formula for a general $L$-function. We refer the readers to \cite[Section 3]{rubinstein} for a detailed discussion of the approximate functional equation based algorithms.

Our algorithm however starts with a `geometric approximate functional equation'. Which means that we write $L(f,1/2+iT)$ as an integral over a `nice' curve of hyperbolic length $O(T^{1+o(1)})$ up to any given polynomial error. This corresponds to the usual discrete sum form of the approximate functional equation. The derivation of the approximate functional equations used here have geometric motivation behind them. Hence we call these equations as ``geometric approximate functional equations''. The `geometric approximate functional equation' for $L(f,s)$ is given (for holomorphic and Maass forms respectively) by \eqref{geomfehol} and \eqref{geomfemaass}.

Let $\Gamma$ be a lattice of $\Slz$ and let $f$ be a modular (holomorphic or Maass) cusp form of weight $k$. In the Maass form case, let $f$ be an eigenfunction of the hyperbolic laplacian $\Delta=-y^2(\partial_x^2+\partial_y^2)$ with eigenvalue $1/4+r^2 $ on $\gh$. Here $r$ is any positive real number or it is a purely imaginary number with $|r|\leq 1/2$.

We will further assume that $f$ is an eigenfunction of some Fricke involution, as is the case for all the newforms. This will quantify the exponential decay at zero. This implies that there exists a positive constant $C_1$ and a real constant $C_2$ such that:
\begin{equation}
 \label{Fricke}
f(-\frac{C_1}{z})=C_2 z^k f(z).
\end{equation}

In this paper, for simplicity let us assume that $f$ is either holomorphic or an even Maass form. The algorithms will be analogous for the odd Maass cusp form case. For an even Maass cusp form, we will use the following power series expansion :

\begin{equation}
\label{Whittaker eqeven}
f( z) =\sum_{n> 0}\hat{f}(n)W_{r}(nz).
\end{equation}
Here $W_{r}(x+iy)=2\sqrt{y}K_{ir}(2\pi y) \cos (2\pi x) $. Here $K_{ir}$ denotes the $K$-Bessel function. See \cite{iwaniec} for more details. The explicit Fourier expansion for the holomorphic cusp forms is given by
\begin{equation}
 \label{Whittaker hol}
f(z) =\sum_{n> 0}\hat{f}(n)e(nz).
\end{equation}
Given an (even Maass or holomorphic) cusp form $f$ of weight $k$, let us define the $L$ function by
\begin{definition}
 \label{L function}
$L(f,s)=\sum_{n=1}^\infty \frac{\hat{f}(n)}{n^{s+(k-1)/2}}$.
\end{definition}
The usual integral representation for $L(f,s)$, corresponding to a holomorphic cusp form is given by

\begin{equation}
\label{functional equation}
\int_0^\infty f(it)t^{s+(k-3)/2}dt=(2\pi)^{-s-(k-1)/2}L(f,s)\Gamma\left(s+(k-1)/2\right).
\end{equation}
Using the condition \eqref{Fricke} and the exponential decay of $f$ at $\infty$, the integral on the left hand side of \eqref{functional equation} converges for all $s$ in $\mathbb C$. Hence we can use $\eqref{functional equation}$ to compute the values of the $L$ functions on the critical line. However, the gamma factor on the right hand side of \eqref{functional equation} is decaying exponentially with $T$ (here $T=\im(s)$), therefore we need to compute the integral to a very high precision in order to get moderate accuracy in the computations of the $L$-values. This problem has been tackled before in several ways. A solution to this has been suggested in \cite{lagarias} and worked out by Rubinstein in \cite[Chapter 3]{rubinsteinthesis}.

We will proceed in a different fashion. We change the contour of integration that will add an exponential factor which will take care of the exponential decay of the gamma function. We will derive a `geometric approximate functional equation' using this idea in this section.
\subsection{A geometric approximate functional equation for $L(f,s) $}

Let $s=1/2+iT$ and $\alpha=-1+\frac{i}{T}.$

As stated before, we change the contour of integration in \eqref{functional equation} so as to add an exponential factor which will take care of the exponential decay of the gamma function. The case of holomorphic modular forms is easier to deal with, hence we will consider it first.

\paragraph{Holomorphic case}: Let $f$ be a holomorphic cusp form of weight $k>0$.
\begin{align}
 \label{holo}&\int_0^\infty f(\alpha t)t^{s+(k-1)/2-1}dt\\&=
\sum_{n=1}^\infty \hat{f}(n) \int_0^\infty\exp (2\pi i n \alpha t)t^{s+(k-1)/2-1}dt; \nonumber\\&=(2\pi)^{-(s+(k-1)/2)}\sum_{n= 1}^{\infty}\frac{\hat{f}(n)}{n^{s+(k-1)/2}} \frac{1}{(-\alpha i)^{s+(k-1)/2}}\int_{z=-i\alpha \mathbb R_+} \exp(-z)z^{s+(k-3)/2} dz\nonumber\\&=(2\pi)^{-(s+(k-1)/2)}\sum_{n= 0}^{\infty}\frac{\hat{f}(n)}{n^{s+(k-1)/2}} \frac{1}{(-\alpha i)^{s+(k-1)/2}}\int_{z= \mathbb R_+} \exp(-z)z^{s+(k-3)/2} dz\nonumber\\&=(2\pi)^{-(s+(k-1)/2)}\frac{L(f,s)}{(-\alpha i)^{s+(k-1)/2}}\int_0^\infty \exp(-t)t^{s+(k-3)/2}dt\nonumber\\&=(2\pi)^{-(s+(k-1)/2)}\frac{L(f,s)}{(-\alpha i)^{s+(k-1)/2}}\Gamma(s+(k-1)/2).\nonumber
\end{align}

Here, the branch cut for the logarithm is taken along the negative real axis.
Hence  the argument of the logarithm takes values between $-\pi$ and $\pi$. With this choice of the logarithm, the factor $ \frac{1}{(-\alpha i)^{s+(k-1)/2}}$ grows like $O(e^{\pi T/2})$ as $T\rightarrow \infty$, to compensate for the exponential decay of $\Gamma(s+(k-1)/2) $.

\paragraph{Case of Maass forms} :
 We use a similar idea for the non-holomorphic case.  let $$\alpha_1=-1+\frac{i}{T_1}.$$ We will choose a suitable value for $T_1 $ later.

\begin{align*}
\label{L function integral} \int_0^\infty f(\alpha_1 t)t^{s-3/2}dy &=\sum_{n> 0}\hat{f}(n) \int_{0}^{\infty} W_{r}(n\alpha_1 t)t^{s-3/2} dt; \\&=\sum_{n> 0}\frac{\hat{f}(n)}{n^{s-1/2}} \int_{0}^{\infty} W_{r}(\alpha_1 t)t^{s-3/2} dt; \\ &= \sum_{n> 0}\frac{\hat{f}(n)}{n^{s-1/2}} \int_{0}^{\infty} W_{r}((-1+\frac{i}{T_1}) t)t^{s-3/2} dt;\\ &=2T_1^{-\one}\sum_{n> 0}\frac{\hat{f}(n)}{n^{s-1/2}} \int_{0}^{\infty} \cos(2\pi t)K_{ir}(2\pi \frac{t}{ T_1})t^{s-1} dt;\\&=\frac{2T_1^{s-1/2}}{(2\pi)^{s}} \sum_{n> 0}\frac{\hat{f}(n)}{n^{s-1/2}}\int_0^\infty \cos( T_1t)K_{ir}( t)t^{s-1} dt.
\end{align*}
We summarize the result as:
\begin{equation}
\label{equation0}
\int_0^\infty f(\alpha_1 t)t^{s-3/2}dy=\frac{2T_1^{s-1/2}}{(2\pi)^{s}}L(f,s)\int_0^\infty \cos( T_1t)K_{ir}( t)t^{iT-1/2} dt.
\end{equation}

In lemma \ref{final bound} we will show that we can choose $T_1=O(T) $ such that  the integral on the right hand side of \eqref{equation0} is not too small.

\begin{lemma}
\label{final bound}
Let $r$ be any fixed complex number such that $|\im(r)|< 1/2 $ and $T,T_1>0$. Then,
\begin{align}
\label{expintegral} \int_0^\infty \cos( T_1t)K_{ir}( t)t^{iT-1/2} dt&= 2^{iT-3/2}\Gamma(\frac{ir+iT+\one}{2})\Gamma(\frac{-ir+iT+\one}{2})\\&F(\frac{ir+iT+\one}{2},\frac{-ir+iT+\one}{2},\one,-T_1^2 ).\nonumber
\end{align} Here $F$ denotes the hypergeometric function. Moreover there exists a computable constant $C'=O(1) $ depending on $r$ such that for $T_1=C'T$ ,
\begin{align*}\Gamma(\frac{ir+iT+\one}{2})\Gamma(\frac{-ir+iT+\one}{2})F(\frac{ir+iT+\one}{2},\frac{-ir+iT+\one}{2},\one,-T_1^2 )&\\\gg T^{-1}.& \end{align*}
 The implied constant is non zero and depends only on $r$.
\end{lemma}

A proof of the lemma \ref{final bound} can be found in \cite{sarnak1}. However, to keep the paper somewhat self contained we will prove lemma \ref{final bound} in appendix \ref{Bessel}. It will allow us to take $T_1=C'T$ for some $C'=O(1)$. Hence for this choice of $T_1$, once we compute $\int_0^\infty f(\alpha_1 t)t^{s-1}dt $ up to an error of $O(T^{-\gamma}) $, we can compute $L(f,s)$ up to an error of $O(T^{1-\gamma}) $.

We next use the exponential decay of $f$ at $0$ and at $\infty$ to get:
\begin{lemma}
\label{cut the sum}
Given any cusp form $f$, a real number $s_0$ and positive reals $T,\gamma$, if $\alpha=-1+i/T$ and $s=s_0+iT$ then there exists a computable constant $c=O(1+\gamma)$, independent of $T$ such that for any $d_1,d_2\geq c$:
 \begin{equation}
\label{exponential decay}
\int_0^\infty f(\alpha t)t^{s-1}dt=\int_{\frac{1}{d_1T\log T}}^{d_2T \log T} f(\alpha t)t^{s-1}dt+O(T^{-\gamma}).
\end{equation}
\end{lemma}
\begin{proof}
 Using the exponential decay of $f$ at $\infty$, we get that for $t\geq 1 $, $$|f(\alpha t)|=|f(-t+it/T)|\ll \exp(-at/T) $$ for some positive constant $a$.
 Hence for all $t>T$, we have $$|f(\alpha t )t^{s-1}|\leq t^{s_0-1}\exp(-at/T) .$$ Hence,
\begin{align*}
 t^{s_0-1}<\exp(at/2T)\Leftrightarrow (s_0-1)\log t<at/2T\Leftrightarrow 2T(s_0-1)\log t/a<t.
\end{align*}
It is easy to see that there exists a constant $c'$ depending only on $a$ and $s_0 $ such that the above condition holds for $t>t_0=c'T\log T$.
This implies that for $t_1\geq t_0$, \begin{equation}\label{bounddd1}|\int_{t_1}^{\infty}f(\alpha t)t^{s-1}dt|\ll \int_{t_1}^\infty \exp(-at/2T)dt=(2T/a)\exp(-a t_1/2T). \end{equation} Hence for any  $c_1\geq 2(c'+1)(1+1/a)(1+\gamma) $, we get that $$\int_{c_1T\log T}^{\infty}|f(\alpha t)t^{s-1}|dt\ll O(T^{-\gamma}). $$

Similarly using \eqref{Fricke} , we get that for $t<1<T$, we have $$|f(\alpha t)|=|f(-t+\frac{it}{T}) |=\frac{|\alpha|^{-k}}{|C_2|} t^{-k} |f(-C_1/( -t+\frac{it}{T}))|\ll t^{-k}\exp(-C_1a/2Tt).$$ We deal with this case exactly as in the previous case to get a suitable constant $c$.
\end{proof}

Applying Lemma \ref{cut the sum} to \eqref{holo} we get that given any $T,\gamma>0 $ there exists a computable constant $c$ such that for any $t_0\leq \frac{1}{cT\log T}$, and for any $t_1\geq cT\log T$,

\begin{equation}
\label{geomfehol}
 L(f,s)=\frac{(2\pi)^{s+(k-1)/2}(-\alpha i)^{s+(k-1)/2}}{\Gamma(s+(k-1)/2)}\int_{t_0}^{t_1} f(\alpha t)t^{s+(k-3)/2}dt+O(T^{-\gamma}).
\end{equation}
and similarly after applying Lemma \ref{cut the sum} and lemma \ref{final bound} to \eqref{equation0} we get that given any $T,\gamma>0 $ there exist computable constants $c$ and $C'=O(1) $ such that if $T_1=C' T $ then for any $t_0\leq \frac{1}{cT\log T}$, and for any $t_1\geq cT\log T$,

\begin{equation}
\label{geomfemaass}
L(f,s)=\frac{(2\pi)^{s}}{2T_1^{s-1/2}C(T_1)}\int_{t_0}^{t_1} f(\alpha_1 t)t^{s-3/2}dt +O(T^{-\gamma}).
\end{equation}

Here, $C(T_1)=\int_0^\infty \cos(T_1t )K_{ir}( t)t^{s-1} dt $ and $T_1$ is chosen such that $C(T_1)\gg T^{-1}$. Notice that the integrals on the right hand side of \eqref{geomfehol} and \eqref{geomfemaass} can be taken over curves of hyperbolic length $\approx T\log T$. We call \eqref{geomfehol} and \eqref{geomfemaass} as geometric approximate functional equations for $L$-functions corresponding to the holomorphic and the Maass forms respectively. It is easy to see that the integrals on the right hand side of the equations $\eqref{geomfehol}$ and $\eqref{geomfemaass} $ can be computed directly in $O(T^{1+\epsilon}) $ time, given any $\epsilon>0 $. Our algorithms for computing $L(f,s)$ start with these ``geometric approximate functional equations''.
\section{Proof of theorem \ref{main theorem 1}}
\label{section3}
Our main aim is to prove theorems \ref{main theorem 2} and \ref{main theorem 1}. Both the algorithms are very similar and they use basically the same idea but the algorithm for theorem \ref{main theorem 1} is marginally simpler. Hence we discuss it first and in the next section we will deal with theorem \ref{main theorem 2}.

Let $\Gamma$ be a lattice in $\Sl $. Given a real $t$, let $$a(t)=\left(\begin {array}{cc} e^{\frac{t}{2}}&0\\0&e^{-\frac{t}{2}} \end{array} \right),$$ and $$ n(t)=\left(\begin {array}{cc}1&t\\0&1 \end{array} \right).$$ Let $f$ be a holomorphic cusp form on $\gsl$ of weight $k$. In this section we will only discuss the holomorphic case. The algorithm in the case of a Maass form $f$ will be analogous. Recall the lift $\tf:\gsl\rightarrow \mathbb C$ of $f$ defined in \eqref{tf} as: $$\tilde{f}: \left(\begin {array}{cc} a&b\\c&d\\ \end{array} \right)\rightarrow (ci+d)^{-k}f\left(\frac{ai+b}{ci+d}\right) .$$  Recall that we have assumed that $\tf$ is ``well behaved''. \textit{i.e.} that $\tf $ has bounded derivatives.
The Fourier expansion is given by \eqref{Whittaker hol}:
\[f(z) =\sum_{n> 0}\hat{f}(n)e(nz).
 \]
 Hence we get that for any positive integer $T$, \begin{align}e^{-2\pi}\hat{f}(T)&=\int_0^1 f(x+i/T)e(-Tx)dx\nonumber\\&=\int_0^T T^{k/2-1}\tf(a(-\log{T})n(t))e(-t)dt. \label{fourier trans}\end{align}

Notice that the integral on the right hand side of \eqref{fourier trans} is an integral of a well behaved smooth function on a horocycle of length $T$. Hence \textit{a priori} we can `compute' it in $O(T^{1+o(1)}) $ time up to the error at most $O(T^{-\gamma})$, for any given $\gamma$. This will denote the corresponding (to \eqref{geomfehol} and \eqref{geomfemaass}) ``geometric functional equation'' in this case. We will now explain a method to compute the $T^{th}$ Fourier coefficient faster than $O(T)$.

Let $\eta$ be any positive number. We will write the integral \eqref{fourier trans} as a sum of integrals over horocycles of length $T^\eta$ each. For simplicity let's assume $T^\eta$ is an integer. Hence we have,

\begin{align}
\label{sum integral 1}
\int_{0}^T \tf(x_0 n(t))e(-t)dt&=\sum_{j=0}^{\lfloor T^{1-\eta}\rfloor-1} \int_{0}^{T^\eta} \tf(x_o n(jT^\eta+t))e(-t)dt\\&+\int_{\lfloor T^{1-\eta}\rfloor T^{\eta}}^T\tf(x_o n(t))e(-t)dt.\nonumber
\end{align}
Here $x_0=a(-\log T)$. The second integral on the right hand side of \eqref{sum integral 1} is an integral of a smooth well behaved function on a horocycle of length at most $ T^{\eta} $. Hence given $\gamma,\epsilon>0$, we can compute it up to an error of $O(T^{-\gamma}) $ in time $O(T^{\eta+\epsilon}) $, using proposition \ref{integralsum}. In practice, $\epsilon$ will be a fixed ``small'' real number and $\eta $ will eventually be chosen to be $1/8$.

Let $$M=T^\eta.$$

Let us define $I_l(x,\tf)$ by
\begin{definition}
Given smooth function $g$ on $\gsl$, $x$ in $\gsl$ and a non negative integer $l$ we define
\begin{equation}
\label{definition I}
I_l(x,g)=\int_{0}^{M} t^l g(xn(t))e(-t)dt.
\end{equation}

\end{definition}

Hence we can rewrite the equation \eqref{sum integral 1} as

\begin{align}
\int_{0}^T \tf(x_0 n(t))e(-t)dt&=\sum_{j=0}^{\lfloor T^{1-\eta}\rfloor-1} I_0(x_0n(jT^\eta),\tf)\nonumber\\&+\int_{\lfloor T^{1-\eta}\rfloor T^{\eta}}^T\tf(x_o n(t))e(-t)dt.\nonumber
\end{align}
As mentioned in subsection \ref{outline}, the slow divergence of the horocycle flow will imply that a lot of the pieces (in \eqref{sum integral 1}) of the horocyle will be very close to each other. We will group the pieces, which stay ``very close'' to each other and try to compute the integrals on all the pieces in each group ``in parallel''. 

The slow divergence property of the horocycle flow used here, is quantified in the following lemma \ref{epsilon'}.

\begin{lemma}
\label{epsilon'}
Given any $\epsilon>0$ and $\eta>0$, for any $\eta'\geq 2\eta+\epsilon $ and for any x,y such that $x^{-1}y\in U_{T^{-\eta'}}$, we have a constant $c$, independent of $\eta $ and $\eta'$ such that $ yn(t)\in xn(t) U_{cT^{-\epsilon}} \text{ for all } 0\leq t\leq T^\eta$.
\end{lemma}

Let $\epsilon$ be any positive number. We take points \{$x_0n(jM),0\leq j < T^{1-\eta}\}$ and ``reduce'' the points to an approximate fundamental domain. Then we ``sort'' the reduced points into the sets $S_1,S_2,...,S_{N}$ such that $x,y\in S_j\Rightarrow x^{-1}y\in U_{T^{-(2\eta+\epsilon)}}$. For each $j$, let us choose a representative $v_j$ from each $S_j$.

It is easy to see that $N\ll T^{6\eta+3\epsilon}\log T$.

If $y\in S_i$, then lemma \ref{epsilon'} implies that the points $yn(t)$ and $v_in(t)$ stay ``very close'' to each other for all $0\leq t\leq T^{\eta}$. Therefore we use the power series expansion around $v_i n(t)$ to compute values of $\tf$ at $y n(t) $. Hence, we get the following lemma:

\begin{lemma}
\label{lemma 1}
Given $\gamma>0,\epsilon>0$, any $\eta>0$ and $x,y \in S_i $ for some $i$, then we have constants $c_{x,y,\beta,l}$ and $d$ such that
\begin{equation}\label{equation sum}I_0(y,\tf)=\sum_{|\beta|\leq d,\beta\in \mathbb Z_+^3}\sum_{l=0}^{d}\frac{c_{x,y,\beta,l}}{\beta !}I_l(x,\partial^\beta \tf)+O( T^{-\gamma}).\end{equation}
Here, $d=O((1+\gamma)/\epsilon) $. The constant involved in \eqref{equation sum} is polynomial in $d$.
\end{lemma}

Notice that the equation \eqref{equation sum} is the explicit form of \eqref{sumFF} in this case.
Let us observe that the integrals involved on the right hand side of \eqref{equation sum} depend only on $x$. This lemma will allow us to compute the sum $I_0(y,f) $ in parallel for all the points $y$ in $S_i$ in $O(|S_i|+M^{1+\epsilon})$ steps.

Lemma \ref{lemma 1} implies that we can get a number $d=O((1+\gamma)/\epsilon)$ such that
\begin{align}
 \int_{0}^T \tf(x_0 n(t))e(-t)dt&=\sum_{m=1}^N\sum_{y\in S_m} \sum_{|\beta|\leq d,\beta\in \mathbb Z_+^3}\sum_{l=0}^{d}\frac{c_{v_m,y,\beta,l}}{\beta !}I_l(v_m,\partial^\beta \tf)\nonumber\\&+\int_{\lfloor T^{1-\eta}\rfloor T^{\eta}}^T\tf(x_o n(t))e(-t)dt +O(T^{-\gamma}).\label{fourier fe}
\end{align}

We will prove the lemmas \ref{epsilon'} and \ref{lemma 1} in section \ref{epsilonpower}. Let us complete the proof of theorem \ref{main theorem 1} using lemma \ref{lemma 1}.
\begin{proof}[Proof of theorem \ref{main theorem 1}]
Notice that \eqref{fourier fe}, along with \eqref{fourier trans} give us a method of computing $\hat{f}(T)$ up to $O(T^{-\gamma})$ error. Let us compute the time spent in this algorithm.

It is easy to see that ``reducing'' each $x_i$ to an approximate fundamental domain, can be done in $O(\log T)$ time. There are many standard reduction algorithms available. We refer the readers to \cite{voight} for a form of the reduction algorithm. The whole ``reducing'' and sorting process has also been discussed in detail in \cite[chapter 7 ]{vishe}.

As mentioned earlier, $N$, the number of sets $\{S_i\}$  is $\approx$ $O(T^{6\eta+3\epsilon}\log T)$. Hence the total time needed in reducing all the $x_i$'s to the fundamental domain, then sorting them into sets $S_j$'s and picking a representative $v_j$ from each $S_j$ requires $O((T^{1-\eta}+T^{6\eta+3\epsilon})\log T)$ steps.

Notice that in equation \eqref{fourier fe}, the integrals $I_l(v_m,\partial^\beta \tf)$ are independent of the choice of $y$. For each $y\in S_m$, there are $O(1)$ terms in the right hand side of equation \eqref{equation sum}. In section \ref{cxybl} we will show that for fixed $y$ and $v_m$, we can compute each $c_{v_m,y,\beta,l}$'s in $O(1)$ time. Hence we can precompute the constants $c_{v_m,y,\beta,l}, $ for all $y$ in $S_m$, and for all $m$ in $O(T^{1-\eta})$ steps. The constants involved are polynomial in $(1+\gamma)/\epsilon$.

Recall that $M=T^\eta$. For a fixed $v_m$, computing $I_l(v_m,\partial^\beta f) $ for all  $|\beta|,l<d$ takes $O(T^{\eta+\epsilon}) $ operations (using proposition \ref{integralsum}). Hence, using Lemma \eqref{lemma 1}, we need $|S_m|$ more operations to compute $I_0(y,f)$  for all $ y\in S_m$. The maximum number of the sets $S_m's$ is $O({T^{6\eta+3\epsilon}}\log T )$. Hence the total time required to compute $I_0(x_i,f)$ for all $x_i$ is (up to a polynomial factor in $(1+\gamma)/\epsilon$) at most $$\sum_{m=1}^N( T^{\eta+\epsilon}+|S_m|)\ll T^{\eta+\epsilon}T^{6\eta+3\epsilon}+T^{1-\eta}. $$  Notice that the extra $\log$ factors can be absorbed at the end in the exponent $\epsilon$.

The second integral on the right hand side of \eqref{sum integral 1}, can be computed up to an error of at most $O(T^{-\gamma}) $, using at most $O(T^{\eta+\epsilon}) $ operations.  Hence the total number of operations needed to compute $\int_{0}^T \tf(x_o n(t))e(t)dt $ up to an error of $O(T^{-\gamma}) $ is
\begin{equation}\label{time1}O((T^{\eta+\epsilon} T^{6\eta+3\epsilon}+T^{1-\eta} )).\end{equation} The optimal value for $\eta$ is $1/8$.
This implies that given any $\epsilon,\gamma>0$ and real $s_0$, $\hat{f}(T) $ can be computed up to an error at most $O(T^{-\gamma}) $ using at most  $O(T^{7/8+4\epsilon} )$ operations. The constant involved is a polynomial in $(1+\gamma)/\epsilon$.
\end{proof}
\section{Proof of theorem \ref{main theorem 2}}
\label{section4}
The proof of this theorem uses an idea very similar to the proof of theorem $\ref{main theorem 1}$. Let $\Gamma$ be a lattice of $\Slz$. Let $f $ be a (holomorphic or Maass) cusp form on $\gh$. Let $\alpha=-1+i/T $ and $c>0$ be the constant in \eqref{geomfehol}(or in \eqref{geomfemaass} for the Maass form case). Let $s=s_0+iT$ and $t_0=c T\log T$. Let $s_0$ be some fixed real number.   Using the ``geometric approximate functional equations'' \eqref{geomfehol} and \eqref{geomfemaass}, it is enough to give an algorithm to compute $$\int_{t_0}^{t_1} f(\alpha t)t^{s-1}dt, $$ for some suitable $t_1\geq cT\log T$. Recall that in the Maass form case, we need an algorithm to compute $\int_{t_0}^{t_1} f(\alpha_1 t)t^{s-1}dt$. However, the algorithm for computing $\int_{t_0}^{t_1} f(\alpha t)t^{s-1}dt$ can be trivially generalized to an algorithm for computing $\int_{t_0}^{t_1} f(\alpha_1 t)t^{s-1}dt$. Therefore, in this section we will only give an algorithm to compute $\int_{t_0}^{t_1} f(\alpha t)t^{s-1}dt$.

Let $\eta$ be any non-negative number. We proceed in a similar manner as in the previous section and write the above integral as a sum of integrals over segments of hyperbolic length $\approx T^{\eta}$.  In order to do so, let us first define the following quantities:
\begin{definition}
\label{xj1}
$$y_0=t_0\alpha,$$
 $$b_j=(1+T^{-1+\eta})^jt_0, $$
$$y_j=b_j\alpha. $$
\end{definition}
Let $M_1$ be an integer such that $M_1=O( T^{1-\eta}\log T)$ and $$b_{M_1}\leq cT\log T\leq b_{M_1+1} $$ Let $$t_1=b_{M_1+1}. $$

Here $y_0$ denotes the starting point of the approximate horocycle. The points $y_i$ denote the starting points of segments. The hyperbolic distance of the segment between $y_i$ and $y_{i+1}$ is $\approx T^\eta$. $M_1$ is the total number of segments required.

Recall that $c$ is chosen such that it satisfies conditions in \eqref{geomfehol}/\eqref{geomfemaass}. Hence it is enough to give an algorithm to compute $$\int_{t_0}^{t_1}f(\alpha t)t^{s-1}dt $$ up to an error at most $O(T^{-\gamma})$.
\begin{align}
\int_{t_0}^{t_1}f(\alpha t) t^{s-1}dt&=\sum_{j=0}^{M_1}\int_{0}^{b_j T^{-1+\eta}}f(\alpha(b_j+t))(b_j+t)^{s-1}dt\nonumber\\ &=\sum_{j=0}^{M_1}b_j^{s}\int_{0}^{ T^{-1+\eta}}f(\alpha(b_j(1+t))(1+t)^{s-1}dt.
\label{eqnsum}
\end{align}
Let us change the variable to $u=tT$. Hence we can rewrite \eqref{eqnsum} as
\begin{align}
\int_{t_0}^{t_1}f(\alpha t) t^{s-1}dt=T^{-1}\sum_{j=0}^{M_1}b_j^{s}\int_{0}^{M}f(\alpha(b_j(1+u/T))(1+u/T)^{s-1}du.\label{eqnsum1}
\end{align}
Here $M=T^\eta$.
Let
\begin{equation}
\label{kappa definition}
\kappa(t)=\left(\begin {array}{cc} (t/T)^{\frac{1}{2}}&-(tT)^\one\\0&(t/T)^{- \frac{1}{2}}\\ \end{array} \right).  
\end{equation}
$\kappa(t)$ denotes a lift of the curve $\{\alpha t\} $ to $\Sl$.
Notice that $$\kappa(b_j)^{-1}\kappa(b_j(1+u/T))=\left(\begin {array}{cc} (1+u/T)^{1/2}&-u(1+u/T)^{-1/2}\\0&(1+u/T)^{-1/2}\\ \end{array} \right)$$ and that it is independent of $j$. Let
 \begin{equation}
\omega(u)=\left(\begin {array}{cc} (1+u/T)^{1/2}&-u(1+u/T)^{-1/2}\\0&(1+u/T)^{-1/2}\\ \end{array} \right).
\label{kappaj}
\end{equation}
 Let $x_j=\kappa(b_j)$ be a lift of the point $y_j$ to $\Sl$.
 We rewrite equation \eqref{eqnsum1} as

\begin{align}
&\int_{t_0}^{t_1}f(\alpha t) t^{s-1}dt\nonumber\\&=T^{-1}\sum_{j=0}^{M_1}b_j^{s}\int_{0}^{M}f(\alpha(b_j(1+u/T))(1+u/T)^{s-1}du\nonumber\\
&=T^{k/2-1}\sum_{j=0}^{M_1}b_j^{s-k/2}\int_{0}^{T^{\eta}}(1+u/T)^{s-k/2-1}\tf(\kappa(b_j(1+u/T)))dt\nonumber\\
&=T^{k/2-1}\sum_{j=0}^{M_1}b_j^{s-k/2}\int_{0}^{T^{\eta}}(1+u/T)^{s-k/2-1}\tf(\kappa(b_j)\kappa(b_j)^{-1}\kappa(b_j(1+u/T)))dt\nonumber\\&=T^{k/2-1}\sum_{j=0}^{M_1} b_j^{s-k/2}L_0(\tf,x_j)
\label{sum integral 3}
\end{align}

 Here given a smooth function $g$, we define \[L_l(g,x)= \int_{0}^{M} u^l g(x\omega(u)) (1+u/T)^{s-k/2-1}du.\]

We make use of the fact that we are integrating on an ``approximate horocycle''. In particular, that it has slow divergence. We quantify this result in the following proposition. This lemma is analogous to lemma \ref{epsilon'} .

\begin{lemma}
\label{epsilon''}
Let $\eta>0,\epsilon>0 $ be such that $\epsilon<1-3\eta$, and $x,y\in \Sl $ such that $x^{-1}y\in U_{T^{-2
\eta-\epsilon}}$, then we have a constant $c'$ such that \[y\omega(u)\in x\omega(u) U_{c'T^{-\epsilon}} \text{ for all } 0\leq u\leq T^{\eta}.\] $c'$ can be chosen independent of $\eta $ and $T$.
\end{lemma}

The lemma \ref{epsilon''} will be proved in section \ref{epsilonpower}. Let's `reduce and sort' the points $\{x_j\}$ into $N=O( T^{6\eta+3\epsilon}\log T )$ groups $S_1,...,S_N $ such that $$x,y\in S_i\Rightarrow x^{-1}y\in U_{T^{-2\eta-\epsilon}} .$$ It is again easy to see that the number of the sets $S_i$'s is $O(T^{6\eta+3\epsilon}\log T)$. Let us also choose fixed representatives $v_i$ from each $S_i$. For each group $S_i$, let us try to compute the inner integrals (in \eqref{sum integral 3}) corresponding to the points in $S_i$ ``in parallel''.

Let $x,y\in S_i$, then we use power series expansion around points $x\omega(t)$ to compute $\tf(y\omega(t))$. Hence we get,
\begin{lemma}
\label{power series simplification}
Given any $\epsilon,\eta,\gamma>0$, such that $\epsilon<1-3\eta$, and $x,y\in S_i $ for some integer $i$, then there exist constants $e_{x,y,\beta,l}$ and $d'$, independent of $T$, such that we can write
\[L_0(\tf,y)=\sum_{|\beta|<d',\beta\in\mathbb Z_+^3}\sum_{l=0}^{d'}\frac{e_{x,y,\beta,l}}{\beta !}L_l(\partial^\beta \tf,x)+O(T^{-\gamma}).\]Here $d'\ll (1+\gamma)/\epsilon$. The constants involved in $O$ are polynomial in $((1+\gamma)/\epsilon)$.
\end{lemma}

Lemmas \ref{epsilon''} and \ref{power series simplification} will be proved in section \ref{epsilonpower}. Lemma \ref{power series simplification} gives the explicit form of \eqref{sumFF} in this context.

 Lemma \ref{power series simplification} implies that there exists a computable constant $d'=O((1+\gamma)/\epsilon)$ such that 

\begin{align}
\int_{t_0}^{t_1}f(\alpha t) t^{s-1}dt=\sum_{m=1}^N\sum_{y\in S_m}a_y \sum_{|\beta|<d',\beta\in\mathbb Z_+^3}\sum_{l=0}^{d'}\frac{e_{v_m,y,\beta,l}}{\beta !}L_l(\partial^\beta \tf,v_m)+O(T^{-\gamma})
\label{sum integral 4}
\end{align}
Here, for each $y$, let $y=x_n$ for some $n$ then the constants $a_y$ are defined by $a_y=T^{k/2-1}b_n^{s-k/2}$.
 The lemma \ref{power series simplification} and \eqref{sum integral 4} convert the problem of $L$-value computation into the problem of computing $L_l(\partial^\beta \tf,v_i)$ for each $v_i$, and for all $l,|\beta|\ll 1$. Recall that \[L_l(\tf,x)= \int_{0}^{T^{\eta}} u^l \tf(x\omega(u)) (1+u/T)^{s-k/2-1}du.\] Notice that each integral $L_l(\partial^\beta \tf,v_i)$ is an integral of $\partial^\beta\tf$ on a segment of hyperbolic length $\approx T^{\eta}$, hence we can compute it up to an error of $O(T^{-\gamma})$ in time $O(T^\eta)$.

More rigorously, we prove:
\begin{lemma}
\label{compute Ll}Given an integer $l\geq 0$, a vector $\beta \in \mathbb Z^3_+$, $x\in \gsl$ and given positive reals $\gamma$ and $\epsilon$, we can compute $L_l(\partial^\beta \tf,x) $ up to an error at most $O(T^{-\gamma}) $ in $O((T^{\eta})^{1+\epsilon})$ time. Here, the constant in $O$ is a polynomial in $l,|\beta|$ and $ (1+\gamma)/\epsilon$.
\end{lemma}
\begin{proof}
Let $g_l$ be the function defined by $g_l(u)=u^l (1+u/T)^{s-k/2-1} $. It is easy to see that there exists a constant $C $, independent of $l$ such that for any $n\in \mathbb N ,$ and for all $0\leq u\leq T^{\eta}$, \begin{equation}\label{g_l bound} |\partial^ng_l(u)|\ll n!(lC)^n(1+u^l) .\end{equation} If we prove that for any $n\in \mathbb N $, fixed $\beta\in \mathbb Z_+^3$ and for any $x\in \gsl$, there exists a constant $D$ independent of $x$ and $l$ such that, \begin{equation}|\partial^n(g_l(u)\partial^\beta \tf(x\omega(u))) |\ll n!(lD)^n(1+u^l) \label{u/T}\end{equation} then using a  proposition \ref{integralsum}, we get the result.

Recall that \begin{align}\omega(u)&=\left(\begin {array}{cc} (1+u/T)^{1/2}&-u(1+u/T)^{-1/2}\\0&(1+u/T)^{-1/2}\\ \end{array} \right)\nonumber\\&=n(-u)a(\log(1+u/T)). \end{align}

Let $0\leq u_0 \leq T^{\eta}$. It is easy to verify that $$\omega(u_0+t)=\omega(u_0)n(-t/(1+u_0/T))a(\log(1+(u_0+t)/T)-\log(1+u_0/T)). $$ Recall that we have assumed that for any $\beta$ in $\mathbb Z_+^3$ and any $x\in \gsl$, $\partial^\beta\tf(x)\ll 1$.  Therefore, we can use the power series expansion for $\tf $ (see \ref{power}) to get \begin{align}\partial^\beta\tf(x\omega(u_0+t))&=\sum_{|\beta'|=0,\beta\in \mathbb Z_+^2\times 0}^\infty \frac{\partial^{\beta'}\partial^{\beta}\tf(x\omega(u_0))}{\beta' !}\times\nonumber\\&\left(-t/(1+u_0/T)\right)^{\beta_1'}\left(\log(1+(u_0+t)/T)-\log(1+u_0/T)\right)^{\beta_2'}. \end{align}
Let us differentiate $n$ times with respect to $t$. Notice that if $\beta_1'+\beta_2'>n$, then $\partial^n|_{t=0}(-t/(1+u_0/T))^{\beta_1'}(\log(1+(u_0+t)/T)-\log(1+u_0/T))^{\beta_2'}=0$. Hence the only contributions come from the terms for which $|\beta'|\leq n$. We use this fact and well behavedness of $\tf$ to get $$\left|\partial^n|_{u=u_0} \partial^\beta\tf(x\omega(u/T))\right|\ll \sum_{|\beta'|\leq n}1\ll n^3.$$ We use this bound along with \eqref{g_l bound} to get \eqref{u/T} and hence the lemma.
\end{proof}

\begin{proof}[Proof of theorem \ref{main theorem 2}]

Using \eqref{sum integral 4}, lemmas \ref{power series simplification} and \ref{compute Ll}, and following exactly similar steps as in the proof of theorem \ref{main theorem 1}, we get the result.
\end{proof}
\section{Lemmas \ref{lemma 1}, \ref{power series simplification} and computing $c_{x,y,\beta,l}$}
\label{epsilonpower}
Lemmas  \ref{epsilon'} and \ref{epsilon''} are analogous. Similarly lemmas \ref{lemma 1} and \ref{power series simplification} are analogous. Our main goal is to prove them in this section.

\noindent Let $x,y$ be any two arbitrary points in $\gsl$ and let $y=xA$ where $A=\left(\begin {array}{cc} p&q\\r&s\\ \end{array} \right) $. Given any $\delta$, let $U_\delta $ be the $\delta$ neighborhood  ball around identity defined in \ref{U definition}. Let $\kappa(t)=\left(\begin {array}{cc} (t/T)^\one&-(Tt)^\one\\0&(t/T)^{-\one}\\ \end{array} \right),n(t)=\left(\begin {array}{cc} 1&t\\0&1\\ \end{array} \right)$.

\noindent We will use the following easily provable fact to prove the lemmas \ref{epsilon'} and \ref{epsilon''}:

\begin{result}
\label{result1}
There exist absolute constants $C_1$ and $C_2$ such that for any $0\leq \delta<C_1$, and any $A$ in $\Sl$ such that $||A-I||_\infty \leq \delta$, we have that $A\in U_{C_2 \delta}$.
\end{result}
Let us now start with the proof of the lemma \ref{epsilon'}.

\begin{proof}[Proof of lemma \ref{epsilon'}:]

The proposition is equivalent to proving that $(n(-t)An(t))\in U_{cT^{-\epsilon}}$ for some $c$. Using fact \ref{result1}, it is enough to prove that $||n(-t)An(t)-I ||_\infty \ll T^{-\epsilon}$ for all $0<t< T^{\eta} $ where $||X||_\infty $ denotes the usual infinity norm. But we have that \begin{equation}\label{1}n(-t)An(t)= \left(\begin {array}{cc} p-tr&(p-s)t-t^2r+q\\r&tr+s\\ \end{array} \right).\end{equation} It is easy to see that for $T$ sufficiently large, if $||A-I||_\infty <T^{-2\eta-\epsilon}$ then $ ||n(-t)An(t)-I ||_\infty \ll T^{-\epsilon}$ for all $0<t< T^{\eta}$. Hence using \ref{result1}, we get the result.

\end{proof}

If we look at the the equation \eqref{1} and compute the $N,A,K$ coordinates for $n(-t)An(t)$, we get the following immediate corollary

\begin{cor}
\label{power series 1}
Using the same notation as in the proof of lemma \ref{epsilon'}, given any $\beta $ in $\mathbb Z_+^3$ and any $k\in \mathbb Z_+ $, we have constants $c_{x,y,\beta,k}$ such that $$|c_{x,y,\beta,k}|\ll 2^k k^3 \beta !T^{-k(\eta+\frac{\epsilon}{2})} $$  and  \[{(yn(t)-xn(t))^\beta}=\sum_{k \in \mathbb Z_+} c_{x,y,\beta,k} t^k. \] This also implies that given any $\gamma,\epsilon>0$, we have a constant $d=O((1+\gamma)/\epsilon) $ such that $\text{ for all } 0\leq t\leq T^{\eta}$, \begin{equation}\label{cxybeta} \frac{(yn(t)-xn(t))^\beta}{\beta!}=\frac{1}{\beta!}\sum_{k=0}^{d} c_{x,y,\beta,k} t^k+O(T^{-\gamma}). \end{equation}
\end{cor}
\begin{proof}
 Given any $\left(\begin {array}{cc}p&q\\r&s\\ \end{array} \right)$ in $\Sl$, we have \begin{equation}\label{NAK}
\left(\begin {array}{cc}p&q\\r&s\\ \end{array} \right)=n( \frac{pr+qs}{r^2+s^2} )a(-\log(r^2+s^2))K(\tan^{-1}(-\frac{r}{s})). \end{equation}
  We use \eqref{NAK} and equation \eqref{1} to get  \begin{align}n(-t)An(t)=&n(\frac{r(p-tr)+(tr+s)((p-s)t-t^2r+q)}{r^2+(tr+s)^2})\nonumber\\&a(-\log(r^2+(tr+s)^2) )K(\tan^{-1}(-\frac{r}{tr+s})). \end{align}  Let $h_{1,\beta_1,x,y}(t)=(\frac{r(p-tr)+(tr+s)((p-s)t-t^2r+q)}{r^2+(tr+s)^2})^{\beta_1}, $ $h_{2,\beta_2,x,y}(t)=(-\log(r^2+(tr+s)^2 ))^{\beta_2},$ and $h_{3,\beta_3,x,y}(t)=(\tan^{-1}(-\frac{r}{tr+s}))^{\beta_3}.$ Hence we have
$$(yn(t)-xn(t))^\beta=h_{1,\beta_1,x,y}(t)h_{2,\beta_2,x,y}(t)h_{3,\beta_3,x,y}(t).$$
 As $||A-I||_\infty\ll T^{-(2\eta+\epsilon)}$, we have that $|r|,|p-s|,|1-p|,|1-s|,|q|\ll T^{-(2\eta+\epsilon)}$. Using these bounds, it is clear that \begin{equation}\label{hbound}h_{1,\beta_1,x,y}^{(n)}(0)\ll \beta_1!n!2^nT^{-(\eta+\epsilon/2)n}.\end{equation}
 Similar results hold for $h_{2,\beta_2,x,y}$ and $h_{3,\beta_3,x,y}$. Taylor series for $h_{1,\beta_2,x,y}$, $h_{2,\beta_2,x,y}$ and $h_{3,\beta_3,x,y}$  gives us that \begin{equation}\label{formula1}c_{x,y,\beta,l}=\sum_{\beta'\in \mathbb Z_+^3,|\beta'|=l} \frac{h_{1,\beta_1,x,y}^{(\beta_1')}(0)h_{2,\beta_2,x,y}^{(\beta_2')}(0)h_{3,\beta_3,x,y}^{(\beta_3')}(0)}{\beta'!}.\end{equation} This gives us an algorithm to compute $c_{x,y,\beta,l}$. Use \eqref{hbound} to see that $$c_{x,y,\beta,l}\ll 2^l l^3 \beta! T^{-(\eta+\epsilon/2)l}. $$
Using this bound for $c_{x,y,\beta,l} $ we can prove that the radius of convergence for the power series $\sum_{k \in \mathbb Z_+} c_{x,y,\beta,k} t^k  $ is at least $O(T^{\epsilon/2+\eta}) $. Hence using the real analyticity of $h_{j,\beta_1,x,y}$ for $j=1,2,3$, we get that for all $0\leq t\ll T^{\eta+\epsilon/2} $, we have $${(yn(t)-xn(t))^\beta}=\sum_{l \in \mathbb Z_+} c_{x,y,\beta,l} t^l.$$ Using this equation and the bounds on $c_{x,y,\beta,l}$, we get \eqref{cxybeta}.
\end{proof}

\begin{proof}[Proof of lemma \ref{epsilon''}]

Recall $$\omega(u)=\left(\begin {array}{cc} (1+u/T)^{1/2}&-u(1+u/T)^{-1/2}\\0&(1+u/T)^{-1/2}\\ \end{array} \right).$$ Notice that $$\omega(u)=n(-u)+Err. $$ Here $||Err||_\infty\ll T^{-1+2\eta}$ for all $0<u<T^{\eta} $ and $$\omega^{-1}(u)=n(u)+Err_2.$$ where $||Err_2||_\infty\ll T^{-1+2\eta}$ for all $0<t<T^{\eta} $.
Using result \ref{result1}, the proposition is equivalent to proving that $$||\omega^{-1}(u)A\omega(u)||_\infty\ll T^{-\epsilon}$$ for all $0<u<T^{\eta}$. However, using the above estimates we get, \[||\omega ^{-1}(u)A\omega(u)-n(u)An(-u)||_\infty\ll T^{{-1+3\eta}}.\]

 Hence for $\epsilon<1-3\eta$, using the same technique as in the proof of lemma \ref{epsilon'}, we get the result.
\end{proof}

Again, if we compute the $NAK$ coordinates of $\omega ^{-1}(u)A\omega(u) $, we get following corollary for any $\beta\in \mathbb Z_+^3$

\begin{cor}
\label{power series 2}
Using the same notation as in the proof of lemma \ref{epsilon''}, given any $\beta $ in $Z_+^3$, we have constants $e_{x,y,\beta,k}$ such that $$|e_{x,y,\beta,k}|\ll 2^k k^3\beta!T^{-k(\eta+\frac{\epsilon}{2})} $$ and \[(y\omega(u)-x\omega(u))^\beta=\sum_{k \in \mathbb Z_+} e_{x,y,\beta,k} u^k. \] This also implies that given any $\gamma,\epsilon>0$, we have a constant $d'=O((1+\gamma)/\epsilon) $ such that $\text{ for all } 0<u<T^{\eta}$, \[ \frac{(y\omega(u)-x\omega(u))^\beta}{\beta!}=\frac{1}{\beta!}\sum_{k=0}^{d'} e_{x,y,\beta,k}u^k +O(T^{-\gamma}). \] 
\end{cor}

\begin{proof}[Proof of lemma \ref{lemma 1}]

Let $y=xA$ be as defined at the beginning of this section, we use power series expansion for $\tf$ around points $xn(t)$ to compute the value of $\tf$ at points $yn(t)$. Using lemma \ref{epsilon'} and the ``well-behavedness'' of $\tf$, we get that $\text{ for all } 0<t<T^{\eta}$, we have a constant $d=O((1+\gamma)/\epsilon)$ such that
\[\tf(yn(t))=\sum_{|\beta|=0}^{d}\frac{\partial^\beta \tf(xn(t))}{\beta !}(yn(t)-x(n(t))^\beta+O(T^{-\gamma}),\] the constant involved in $O$ only depends on $f$. We use corollary \eqref{power series 1} to get that \[\tf(yn(t))= \sum_{|\beta|=0}^{d}\partial^\beta \tf(xn(t))(\frac{1}{\beta!}\sum_{l=0}^{d} c_{x,y,\beta,l} t^l+O(T^{-\gamma}))+O(T^{-\gamma}).\] Hence we have constants $c_{x,y,\beta,l}$ such that \[\tf(yn(t))=\sum_{|\beta|<d,\beta\in\mathbb Z_+^3}\sum_{l=0}^{d}c_{x,y,\beta,l} t^l\frac{\partial^\beta \tf(xn(t))}{\beta !} +O(d^4 T^{-\gamma}). \] Here the $O$ constant only depends on $f$. Now integrating, we get the result.
\end{proof}

\begin{proof}[Proof of lemma \ref{power series simplification}]

 By using exactly the same proof as in the proof of lemma \ref{lemma 1}  we get that we have constants $e_{x,y,\beta,l} $ and $d'=O((1+\gamma)/\epsilon)$ such that
\[\tf(y\omega(u))=\sum_{|\beta|<d',\beta\in\mathbb Z_+^3}\sum_{l=0}^{d'}\frac{e_{x,y,\beta,l}}{\beta !} u^l \partial^\beta \tf(x\omega(u)) +O((d')^4 T^{-\gamma}). \]
Integrating on both sides, we get the result.
\end{proof}

\subsection{Computing $c_{x,y,\beta,l} $ and $e_{x,y,\beta,l}$}
\label{cxybl}
Recall that $d=O((1+\gamma)/\epsilon)$ and we need to compute $c_{x,y,\beta,l}$ for all $|\beta|,l\leq d$. Recall that using \eqref{formula1}, we have
\begin{equation*}c_{x,y,\beta,l}=\sum_{\beta'\in \mathbb Z_+^3,|\beta'|=l} \frac{h_{1,\beta_1,x,y}^{(\beta_1')}(0)h_{2,\beta_2,x,y}^{(\beta_2')}(0)h_{3,\beta_3,x,y}^{(\beta_3')}(0)}{\beta'!}.\end{equation*}
Here $$h_{1,\beta_1,x,y}(t)=(\frac{r(p-tr)+(tr+s)((p-s)t-t^2r+q)}{r^2+(tr+s)^2})^{\beta_1},$$ $$h_{2,\beta_2,x,y}(t)=(-\frac{1}{2}\log(r^2+(tr+s)^2 ))^{\beta_2},$$ and $$h_{3,\beta_3,x,y}(t)=(\tan^{-1}(-\frac{r}{tr+s}))^{\beta_3}.$$

It is easy to see that we can compute $c_{x,y,\beta,l}$ for any $|\beta|,l\leq d$ in $O(1)$ time. The constant here is a polynomial in $d$.

A similar method will go through for computing $e_{x,y,\beta,l}$.

\section{Numerical integration for analytic functions}
\label{numerical integration}
Let us state and prove the following simple algorithm for numerical integration of a real analytic function, used repeatedly in the paper. In practice this integration can be expediated in a number of ways by using a doubly exponential integration or Gaussian quadrature technique for numerical integration. 
\begin{proposition}
\label{integralsum}
Let $T$ be any positive number and let $f$ be a real analytic function on an open set containing $[0,T]$. Let $l$ be a fixed integer and $R$ be a positive constant such that for any $k\in \mathbb N$ of $|\partial^kf(x)|\ll k!R^k (1+x^l) $ for all $x$ in $[0,T]$. If at any given point in $[0,T]$ and $n\in \mathbb N$ and any $\gamma>0$, we can compute $n^{th}$ derivative of $f$ in polynomial (in $n$ ) time, then for any given $\epsilon,\gamma>0$, we can compute $\int_0^T f(t)dt$  up to an error at most $O(T^{-\gamma})$, using at most $O(T^{1+\epsilon})$ arithmetic operations. The constant involved is polynomial in $R,(1+\gamma)/\epsilon$ and $l$.\end{proposition}

\begin{proof}
The idea is to use a fine grid of $T^{1+\epsilon}R$ equispaced points and use power series expansion around the nearest left grid point to compute $f$. In particular, let us split the integral into integrals over intervals each of size $T^{-\epsilon}/R$.

Let $$M=T^{-\epsilon}/R, M_2=\lfloor T/M\rfloor-1.$$ Hence
\begin{equation}
\int_0^{T} f(t)dt=\sum_{x=0}^{M_2}\int_0^M f(xM+t)dt+\int_{M(M_2+1)}^{T}f(t)dt.
\label{suminteg}
\end{equation}

Let us use power series expansion around each $xM$ to compute the value of $f$ at a nearby point. Hence we get
$$f(xM+t)=\sum_{l=0}^\infty \partial^l(f)(xM)\frac{t^l}{l !}.$$
For $|t|<M$ and any fixed non negative integer $N$, we get

\begin{align}
|f(xM+t)-\sum_{n=0}^N \partial^l(f)(xM)\frac{t^n}{n !}|&\ll T^{l}\sum_{n=N+1}^\infty T^{-\epsilon n}\ll T^{l-\epsilon N}.
\label{summ}
\end{align}

Substituting this in equation\eqref{suminteg}, we get
\begin{align}
\int_0^{T} f(t)dt&=\sum_{x=0}^{M_2}\sum_{k=0}^N\int_0^M \partial^k(f)(xM)\frac{t^k}{k !}dt\nonumber \\&+\sum_{k=0}^N\partial^k f(M(M_2+1))\int_{0}^{T-MM_2-M}\frac{t^k}{k !}dt+E\nonumber
\\&=\sum_{x=0}^{M_2}\sum_{k=0}^N \partial^k(f)(xM)\int_0^M\frac{t^k}{k !}dt \nonumber\\&+\sum_{k=0}^N\partial^k f(M(M_2+1))\int_{0}^{T-MM_2-M}\frac{t^k}{k !}dt+E\label{summ1}.
\end{align}
Here, $$|E|\ll (M_2+1) T^{l-N\epsilon}. $$
Notice that any $\partial^k f(x)$ can be computed at any $x$ in polynomial in $k $ time. Notice that the total number of operations needed to compute the sum on the right hand side of \eqref{summ} is $O(M_2N)$. We choose $N=(1+l/\gamma)\frac{\gamma}{\epsilon}$ to get the result.
\end{proof}

\appendix
\section{Special functions and lemma \ref{final bound}}
\label{Bessel}

As mentioned earlier, the proof of lemma \ref{final bound} given here can also be found, in a slight different form in \cite{sarnak1}. To keep the paper somewhat self contained, we will prove it again in this section.

Let us recall the definitions and some properties of some special functions and prove proposition  \ref{final bound}.
\begin{definition}
Given, $a,b,c$ any complex numbers and $|z|<1$, the hypergeometric function $F(a,b,c,z)$ is defined by

\begin{equation}
F(a,b,c,z)=\frac{\Gamma(c)}{\Gamma(a)\Gamma(b)}\sum_{n=0}^\infty \frac{\Gamma(n+a)\Gamma(n+b)}{\Gamma(n+c)n!}z^n.
\label{hypergeometric}
\end{equation}

There is an analytic continuation of this function to the whole complex plane except a branch cut from 1 to infinity.
\end{definition}

We will use the following well known transformation property of Hypergeometric functions (see \cite[(9.132)]{integralssums}).

\begin{align}
F(a,b,c,z)&=\frac{\Gamma(c)\Gamma(b-a)}{\Gamma(b)\Gamma(c-a)}(1-z)^{-a}F(a,c-b,1+a-b,\frac{1}{1-z})\nonumber \\& +\frac{\Gamma(c)\Gamma(a-b)}{\Gamma(a)\Gamma(c-b)}(1-z)^{-b}F(b,c-a,1+b-a,\frac{1}{1-z}).
\label{hypergeom}
\end{align}

We also need the following asymptotic expansion for large arguments of the Gamma function (\cite[8.327]{integralssums}) given as

\begin{equation}
 \Gamma(z)=z^{z-\one}e^{-z}\sqrt{2\pi}(1+O(1/|z|) ).
\end{equation}
Hence for $z=\sigma+it$, $\sigma$ fixed and large $|t|$ we get,
\begin{equation}
\Gamma(\sigma+it)=e^{i\frac{\sgn(t)\pi(\sigma-\one)}{2}}e^{-\frac{\pi |t|}{2}}(\frac{|t|}{e})^{it}|t|^{\sigma}(\frac{2\pi}{|t|})^{\one}(1+O(t^{-1})).
\label{gamma assymptotics}
\end{equation}
 Here, $t$ is positive and the constant in $O $ depends on $\sigma$.
We use the following result about the Bessel functions (\cite[(6.699)]{integralssums}):

\begin{result}
 Let $r$ be any fixed complex number such that $|\im(r)|< \one $ and $T,T'>0$, we have
\begin{align}
\label{expintegral1} \int_0^\infty \cos( T_1t)K_{ir}( t)t^{iT-1/2} dt&= 2^{iT-3/2}\Gamma(\frac{ir+iT+\one}{2})\Gamma(\frac{-ir+iT+\one}{2})\\&F(\frac{ir+iT+\one}{2},\frac{-ir+iT+\one}{2},\one,-T_1^2 ).\nonumber
\end{align}
\end{result}
Let us use \eqref{hypergeom} to get
\begin{align}\label{hypergeom2}
 &F(\frac{ir+iT+\one}{2},\frac{-ir+iT+\one}{2},\one,-T_1^2 )\\&=\frac{\Gamma(\one)\Gamma(-ir)(1+T_1^2)^{\frac{-ir-iT-\one}{2}}}{\Gamma(\frac{-ir+iT+\one}{2})\Gamma(\frac{\one-ir-iT}{2} )}F(\frac{ir+iT+\one}{2},\frac{ir-iT+\one}{2},ir+1,\frac{1}{1+T_1^2} )\nonumber\\&+\frac{\Gamma(\one)\Gamma(ir)(1+T_1^2)^{\frac{ir-iT-\one}{2}}}{\Gamma(\frac{ir+iT+\one}{2})\Gamma(\frac{ir-iT+\one}{2} )}F(\frac{iT-ir+\one}{2},\frac{-ir-iT+\one}{2},-ir+1,\frac{1}{1+T_1^2} ).\nonumber
\end{align}

Now
$$F(a,b,c,z)=1+O(|abz/c|)$$
uniformly for
$$|z|\max_{l\geq 0}|{\frac{(a+l)(b+l)}{(c+l)(l+1)}}|\leq \one$$
(see \cite{good} and \cite{sarnak}). Hence we can get a positive constant $B$
depending only on $r$ such that for any $B'>B $,
both $|F(\frac{ir+iT+\one}{2},\frac{ir-iT+\one}{2},ir+1,\frac{1}{1+B'^2T^2}
)-1|$ and
$|F(\frac{iT-ir+\one}{2},\frac{-ir-iT+\one}{2},-ir+1,\frac{1}{1+B'^2T^2}
)-1|$ are less than or equal to $0.1$.
Applying this and \eqref{hypergeom2} to \eqref{expintegral1} we get a constant $C>B$ such that
for $T_1\geq CT$ we have

\begin{equation}
 \label{expintegral2}\int_0^\infty \cos( T_1t)K_{ir}( t)t^{iT-\one} dt=D(T_1,T)(1+r_1+E(T_1,T)(1+r_2) ).
 \end{equation}
 Here $|r_1|, |r_2|\leq 0.1$, \begin{equation}\label{Ddef}D(T_1,T)=\frac{2^{iT-3/2}\Gamma(\frac{ir+iT+\one}{2})\Gamma(\one)\Gamma(-ir)(1+T_1^2)^{\frac{-ir-iT-\one}{2}}}{\Gamma(\frac{-ir-iT+\one}{2})}.\end{equation}
  and
\begin{equation}
 \label{Edef}
E(T_1,T)=\frac{\Gamma(ir)\Gamma(\frac{-ir+iT+\one}{2})\Gamma(\frac{-ir-iT+\one}{2})(1+T_1^2)^{ir}}{\Gamma(-ir)\Gamma(\frac{ir+iT+\one}{2})\Gamma(\frac{ir-iT+\one}{2} )}.
\end{equation}
Notice that for a fixed $T$ and real $r$, the argument of $\frac{\Gamma(ir)\Gamma(\frac{-ir+iT+\one}{2})\Gamma(\frac{-ir-iT+\one}{2})}{\Gamma(-ir)\Gamma(\frac{ir+iT+\one}{2})\Gamma(\frac{ir-iT+\one}{2} )} $ is fixed. On the other hand $(1+T_1^2)^{ir} $ is a rapidly oscillating function of $T_1$. Hence we can choose a suitable $C'>B$ such that $|1+r_1+E(C_1T,T)(1+r_2)|>\one  $.

\noindent Similarly if $\re(ir)$ is non zero, then using asympotics of the $\Gamma$ function to \eqref{Edef}, we can see that the behaviour of the absolute value of $E(T_1,T) $ is dominated by $(1+T_1^2)^{ir} $, for $T_1>T$. In other words, we can choose a suitable $C'>B$ such that $|1+r_1+E(C'T,T)(1+r_2)|>\one  $. Moreover, it can be computed in $O(1)$ time. See \cite[8.1]{vishe}.

Using the asymptotics of $\Gamma$ function, it is easy to get a simple bound $|D|\gg T^{-1} $. The constant only depends on $r$. We have thus proved lemma \ref{final bound}.

\bibliographystyle{amsalpha}

\end{document}